\documentclass[amsfonts,12pt]{amsart}

\usepackage{epsfig}

\usepackage{amssymb}
\usepackage{graphicx}
\usepackage{stmaryrd}
\usepackage{color}
\usepackage{pifont}
\usepackage{enumerate}

\newcommand{\N}{{\mathbb N}}
\newcommand{\Q}{{\mathbb Q}}
\newcommand{\R}{{\mathbb R}}

\newcommand{\Sn}{S^{n-1}}

\newtheorem{thm}{Theorem}[section]

\newtheorem{lem}[thm]{Lemma}
\newtheorem{cor}[thm]{Corollary}
\newtheorem{rem}[thm]{Remark}

\newtheorem{prob}[thm]{Problem}

\newcommand{\relbd}{{\mathrm{relbd}}\,}
\newcommand{\cl}{{\mathrm{cl}}\,}

\newcommand{\lin}{{\mathrm{lin}}\,}

\makeatletter
\@namedef{subjclassname@2020}{\textup{2020} Mathematics Subject Classification}
\makeatother

\begin{document}
\hfill\today
\bigskip

\title{Full rotational symmetry from reflections or rotational symmetries in finitely many subspaces}
\author[Gabriele Bianchi, Richard J. Gardner, and Paolo Gronchi]
{Gabriele Bianchi, Richard J. Gardner, and Paolo Gronchi}
\address{Dipartimento di Matematica e Informatica ``U. Dini", Universit\`a di Firenze, Viale Morgagni 67/A, Firenze, Italy I-50134} \email{gabriele.bianchi@unifi.it}
\address{Department of Mathematics, Western Washington University,
Bellingham, WA 98225-9063} \email{richard.gardner@wwu.edu}
\address{Dipartimento di Matematica e Informatica ``U. Dini", Universit\`a di Firenze, Piazza Ghiberti 27, Firenze, Italy I-50122} \email{paolo.gronchi@unifi.it}
\thanks{First and third author supported in part by the Gruppo
Nazionale per l'Analisi Matematica, la Probabilit\`a e le loro
Applicazioni (GNAMPA) of the Istituto Nazionale di Alta Matematica (INdAM).  Second author supported in
part by U.S.~National Science Foundation Grant DMS-1402929.}

\maketitle
\pagestyle{myheadings}
\markboth{GABRIELE BIANCHI, RICHARD J. GARDNER, AND PAOLO GRONCHI}{FULL ROTATIONAL SYMMETRY FROM REFLECTIONS OR ROTATIONAL SYMMETRIES}

\begin{abstract}
Two related questions are discussed.  The first is when reflection symmetry in a finite set of $i$-dimensional subspaces, $i\in \{1,\dots,n-1\}$, implies full rotational symmetry, i.e., the closure of the group generated by the reflections equals $O(n)$. For $i=n-1$, this has essentially been solved by Burchard, Chambers, and Dranovski, but new results are obtained for $i\in \{1,\dots,n-2\}$.  The second question, to which an essentially complete answer is given, is when (full) rotational symmetry with respect to a finite set of $i$-dimensional subspaces, $i\in \{1,\dots,n-2\}$, implies full rotational symmetry, i.e., the closure of the group generated by all the rotations about each of the subspaces equals $SO(n)$.  The latter result also shows that a closed set in $\R^n$ that is invariant under rotations about more than one axis must be a union of spheres with their centers at the origin.
\end{abstract}

\section{Introduction}
The question of when symmetry in a restricted set of subspaces implies full rotational symmetry arises naturally in several contexts.  Radial functions are particularly amenable to analysis, and a function on $\R^n$ is radial if and only if it is symmetric with respect to reflection in any hyperplane through the origin.  A basic geometrical or mechanical problem asks whether a body that is rotationally symmetric in more than one axis must be a ball.  Our motivation comes from the widely applied method of Steiner symmetrization, in which iterated Steiner symmetrals of a compact set, if taken with respect to a suitable sequence of hyperplanes, converge to a ball.  For more information about the latter, see \cite{BGG2}, where we find use for all the main results obtained here.

Let $i\in \{1,\dots,n-1\}$.  We say that a set $G$ of reflections in $i$-dimensional subspaces of $\R^n$ generates $O(n)$ if the closure of the group $\langle G\rangle$ generated by $G$ equals $O(n)$.  For the applications we have in mind, it is particularly desirable to find finite sets $G$ with this property, and this is the focus of most of the paper, in Section~3.  In Section~4, we address the corresponding question as to when $G$ generates $SO(n)$ if $G$ is the set of all the rotations about each of several subspaces, each with dimension between $1$ and $n-2$.

In the case of reflections, prior work mainly addresses the case when $i=n-1$, i.e., of reflections in hyperplanes through the origin.  When $G$ is infinite, Eaton and Perlman \cite{EP} proved that $G$ generates $O(n)$ if it is irreducible, i.e., there is no nontrivial proper subspace $S$ of $\R^{n}$ such that $gS\subset S$ for $g\in G$, and it is clear that this sufficient condition is also necessary.  The problem for arbitrary sets $G$ was solved by Burchard, Chambers, and Dranovski \cite{BCD}. In \cite[Proposition~4.2]{BCD}, they show that for a set of reflections to generate $O(n)$, it suffices that the unit normal vectors to the corresponding hyperplanes satisfy three conditions: (i) at least two are at an angle that is an irrational multiple of $\pi$; (ii) they span $\R^n$; and (iii) they cannot be partitioned into two mutually orthogonal subsets. As these authors point out on \cite[p.~1189]{BCD}, the first condition can be replaced by (i$'$): the set of reflections do not belong a finite Coxeter group.  The latter is weaker than (i), since, when $n>2$, it is enough that at least two unit normal vectors are at an angle which is not an integer multiple of $\pi/3$, $\pi/4$, or $\pi/5$.  Conditions (i$'$), (ii), and (iii) are necessary as well as sufficient.

We obtain the following results for reflections.  Theorem~\ref{teo:sphericalsym_lines} deals with the relatively simple case when $i=1$, and proves that the conditions (i), (ii), and (iii) stated above suffice, where the unit vectors are now taken parallel to the subspaces.  When $i\in \{2,\dots,n-2\}$, the situation is inherently more complicated.  One of the main ingredients in our approach is the quite difficult but general Lemma~\ref{extension_lemma}, which we call the symmetry extension lemma.  This assumes that $H$ and $L$ are subspaces of $\R^n$ such that $H \cap L^\perp=\{o\}$ and $\dim H < \dim L$, and that $G$ is a closed subgroup of $O(n)$ that acts transitively on $S^{n-1}\cap (L+x)$ for each $x\in S^{n-1}$.  (Here $+$ denotes the Minkowski or vector sum; see Section~\ref{subsec:notations} for this and other notation and terminology.) The conclusion is that the closure of the subgroup generated by $G$ and the reflection $R_H$ in $H$ acts transitively on $S^{n-1}\cap (H+L+x)$ for each $x\in S^{n-1}$. Using the symmetry extension lemma and other tools such as Kronecker's approximation theorem, we are able in Theorem~\ref{finthm} to provide sufficient conditions for a finite set of reflections in $i$-dimensional subspaces to generate $O(n)$. In Corollary~\ref{cor21June}, we conclude that if
$$
k=
\begin{cases}
n, & \text{if $i=1$ or $i=n-1$,}\\
\lceil n/i\rceil+1, & \text{if $1<i\leq n/2$,}\\
\lceil n/(n-i)\rceil+1, & \text{if $n/2\leq i<n-1$,}
\end{cases}
$$
then there are $i$-dimensional subspaces $H_1,\dots,H_k$, such that the reflections $R_{H_1},\dots,R_{H_k}$ generate $O(n)$, or, equivalently, are
such that if $E\subset \Sn$ is nonempty, closed, and satisfies $R_{H_j} E=E$ for $j=1,\dots, k$, then $E=\Sn$.  This in turn implies that if $K$ is a convex body in $\R^n$ satisfying $R_{H_j}K=K$ for $j=1,\dots, k$, then $K$ is a ball with center at the origin.

The previous result prompts the question as to what happens if reflections in subspaces are replaced by (full) rotational symmetries in subspaces, and we answer this in Section~4.  (See \cite{KK} for a result in $\R^3$ in the same spirit.)  There are easy examples of nonempty, closed, proper subsets of $S^{n-1}$ that are left invariant under all rotations about more than one subspace; for example, the Cartesian product in $\R^4$ of two orthogonal copies of the circle with center at the origin and radius $1/\sqrt{2}$ is invariant under any rotation about each of the two-dimensional subspaces orthogonal to the circles.  Similarly, the Cartesian product in $\R^4$ of two orthogonal disks centered at the origin (a so-called duocylinder) is a nonspherical convex body of revolution about each of the two-dimensional subspaces orthogonal to the disks. Theorem~\ref{new} states a necessary and sufficient pair of conditions for subspaces $H_1,\dots,H_k$ with $1\le \dim H_j\le n-2$ to be such that if a nonempty closed subset of $S^{n-1}$ is invariant under any rotation about each $H_j$, $j=1,\dots,k$, then $E=S^{n-1}$. The latter is equivalent to saying that a closed set in $\R^n$ that is invariant under rotations that fix $H_j$, $j=1,\dots,k$, must be a union of spheres with their centers at the origin, or to saying that if a convex body $K$ in $\R^n$ is rotationally symmetric in $H_j$, $j=1,\dots,k$, then $K$ must be a ball with center at the origin.  The conditions are (a) $H_1^{\perp}+\cdots+H_k^{\perp}=\R^n$ and (b) $\{H_1^{\perp},\dots,H_k^{\perp}\}$ cannot be partitioned into two mutually orthogonal nonempty subsets.   Corollary~\ref{cor22June} draws the conclusion that if $1\leq i\leq n-2$ and $k=\lceil n/(n-i)\rceil$, then there are $i$-dimensional subspaces
$H_1,\dots, H_k$, that satisfy (a) and (b).  Corollary~\ref{corSep26} shows that among closed sets in $\R^n$, $n\ge 3$, only unions of spheres with their centers at the origin are invariant under rotations about two different axes.

We are very grateful to Fabio Podest\`a for valuable discussions, and to a referee who suggested extracting the present paper from the original version of \cite{BGG2} and whose thorough reading and extensive knowledge led to many improvements.

\section{Preliminaries}\label{subsec:notations}

As usual, $S^{n-1}$ denotes the unit sphere and $o$ the origin in Euclidean $n$-space $\R^n$ with Euclidean norm $\|\cdot\|$.  We assume throughout that $n\ge 2$.   The term {\em ball} in $\R^n$ will always mean an $n$-dimensional ball unless otherwise stated. The unit ball in $\R^n$ will be denoted by $B^n$. If $x,y\in \R^n$, we write $x\cdot y$ for the inner product.

Let $X$ and $Y$ be sets in $\R^n$.  We denote by $\lin X$ and $\cl X$ the {\it linear hull} and {\it closure} of $X$, respectively.  If $t\in \R$,
then $tX=\{tx:x\in X\}$ is the {\em dilate} of $X$ by the factor $t$ and
$$X+Y=\{x+y: x\in X, y\in Y\}$$
denotes the {\em Minkowski sum} of $X$ and $Y$.

If $x\in \R^n\setminus\{o\}$, then $x^{\perp}$ is the $(n-1)$-dimensional subspace orthogonal to $x$. Throughout the paper, the term {\em subspace} means a linear subspace.   If $H$ is a subspace of $\R^n$, then $\dim H$ is its {\em dimension}, $X|H$ is the (orthogonal) projection of a set $X$ on $H$, and $x|H$ is the projection of a vector $x\in \R^n$ on $H$.

The Grassmannian of $k$-dimensional subspaces in $\R^n$ is denoted by ${\mathcal{G}}(n,k)$.

As usual, $O(n)$ and $SO(n)$ denote the {\em orthogonal group} and {\em special orthogonal group}, respectively, of isometries of $\R^n$. If $\phi$ is an isometry and $H$ is a subspace of $\R^n$, we say that $\phi$ {\em fixes} $H$ (or $H$ is {\em fixed} by $\phi$) if $\phi$ acts as the identity on $H$.  (Note the difference between this and invariance; as usual, we say that a set $X$ is {\em invariant} under $\phi$ if $\phi X = X$.)  By $O(n)_H$ (or $SO(n)_H$) we mean the {\em pointwise stabilizer} of $H$, i.e. the subgroup of isometries in $O(n)$ (or $SO(n)$, respectively), that fix $H$.  If $G$ is a subgroup of $O(n)$ and $x\in\Sn$, we denote by $G(x)=\{\phi x: \phi\in G\}$ the {\em orbit} of $x$.

If $H$ is a subspace of $\R^n$, we write $R_HX$ for the {\em reflection} of $X$ in $H$, i.e., the image of $X$ under the map that takes $x\in \R^n$ to $2(x|H)-x$.  If $R_HX=X$, we say $X$ is {\em $H$-symmetric}. If $H=\{o\}$, we instead write $-X=(-1)X$ for the reflection of $X$ in the origin and {\it $o$-symmetric} for $\{o\}$-symmetric. A set $X$ is called {\em rotationally symmetric} with respect to the $i$-dimensional subspace $H$ if for all $x\in H$, $X\cap (H^{\perp}+x)$ is a union of $(n-i-1)$-dimensional spheres, each with center at $x$.  These are precisely the sets that are invariant under each element of $O(n)_H$ or of $SO(n)_H$. If $\dim H=n-1$, then $X$ is rotationally symmetric with respect to $H$ if and only if it is $H$-symmetric.

As usual, $\alpha_1,\dots,\alpha_l\in \R$ are called {\em linearly independent over $\mathbb{Q}$} if $\sum_{j=1}^l q_j \alpha_j=0$ with $q_1,\dots,q_l\in \Q$ implies that $q_1=\dots=q_l=0$.

Finally, ${\mathcal K}^n_n$ signifies the class of {\em convex bodies}, i.e., compact convex sets with interior points.  This notation follows that of Schneider's classic text \cite{Sch93}.

\section{Full rotational symmetry via reflections in finitely many subspaces}\label{reflections}

This section focuses on finding finite sets of $i$-dimensional subspaces such that reflections in these subspaces generate full rotational symmetry.

Our first main result, Theorem~\ref{teo:sphericalsym_lines} below, relies essentially on the next lemma, a special case of the symmetry extension lemma, Lemma~\ref{extension_lemma}. It is closely related to \cite[Lemma~4]{MS43}, a result stated using the notion of dimension of a subgroup of a Lie group and proved using group theory.  We provide a simple topological proof.  Note that since $E\subset S^{n-1}$, requiring $E$ to be invariant with respect to every element of $O(n)_H$ is the same as requiring $E$ to be rotationally symmetric with respect to $H$.

\begin{lem}\label{lemma_MS}
Let $H\in {\mathcal{G}}(n,1)$, $n\geq 3$, and let $O(n)_H$ be its pointwise stabilizer in $O(n)$. Let $\phi\in O(n)$ be such that $\phi H\neq H$. If $E\subset S^{n-1}$ is nonempty, closed, and invariant with respect to $\phi$ and to all the elements of $O(n)_H$, then $E=S^{n-1}$.
\end{lem}

\begin{proof}
Suppose that $E\neq S^{n-1}$.  The set $S^{n-1}\setminus H$ contains at least one point in $E$, for otherwise $E\subset H$ and $E$ could not be invariant with respect to $\phi$.  It also contains at least one point not in $E$, or else $E$, being closed, would coincide with $S^{n-1}$.  Since any continuous curve joining a point of $E$ to a point not in $E$ contains a point of $\relbd E$, there exists a point $v\in S^{n-1}\setminus H$ with $v\in \relbd E$, where here and for the remainder of the proof, $\relbd$ denotes the boundary relative to $S^{n-1}$.

Note that if $\psi \in O(n)$ and $\psi E=E$, then $\psi(\relbd E)=\relbd E$.  It follows that the orbit $C=O(n)_H(v)$ of $v$ under $O(n)_H$ is a hyperplane section of $S^{n-1}$ contained in $\relbd E$. It is connected because $n\geq 3$. The image $\phi C$ of $C$ under $\phi$ is also a hyperplane section of $S^{n-1}$ contained in $\relbd E$.  The projection $(\phi C)|H$ of $\phi C$ onto $H$ is an interval $I$, since $\phi C$ is connected.  See Figure~1.  The set $I$ is not a single point, because the hyperplane containing $\phi C$ is not orthogonal to $H$, due to $\phi H\neq H$.
Let
\begin{equation}\label{lemma_MS_1}
D=\{O(n)_H(x) : x\in\phi C\}=\{u\in\Sn : u|H\in I\}.
\end{equation}

Let $D'$ be the interior of $D$ relative to $S^{n-1}$ and let $I'$ be the interior of $I$ relative to $H$. Then \eqref{lemma_MS_1} yields
$D'=\{u\in\Sn : u|H\in I'\}\neq\emptyset$. However, $D\subset \relbd E$, and since $E$ is closed, the interior of $\relbd E$ relative to $S^{n-1}$ is empty, so $D'=\emptyset$, a contradiction.
\end{proof}

\begin{figure}[htb]\label{fig1}
\begin{center}
%\vspace*{1cm}
\epsfig{file=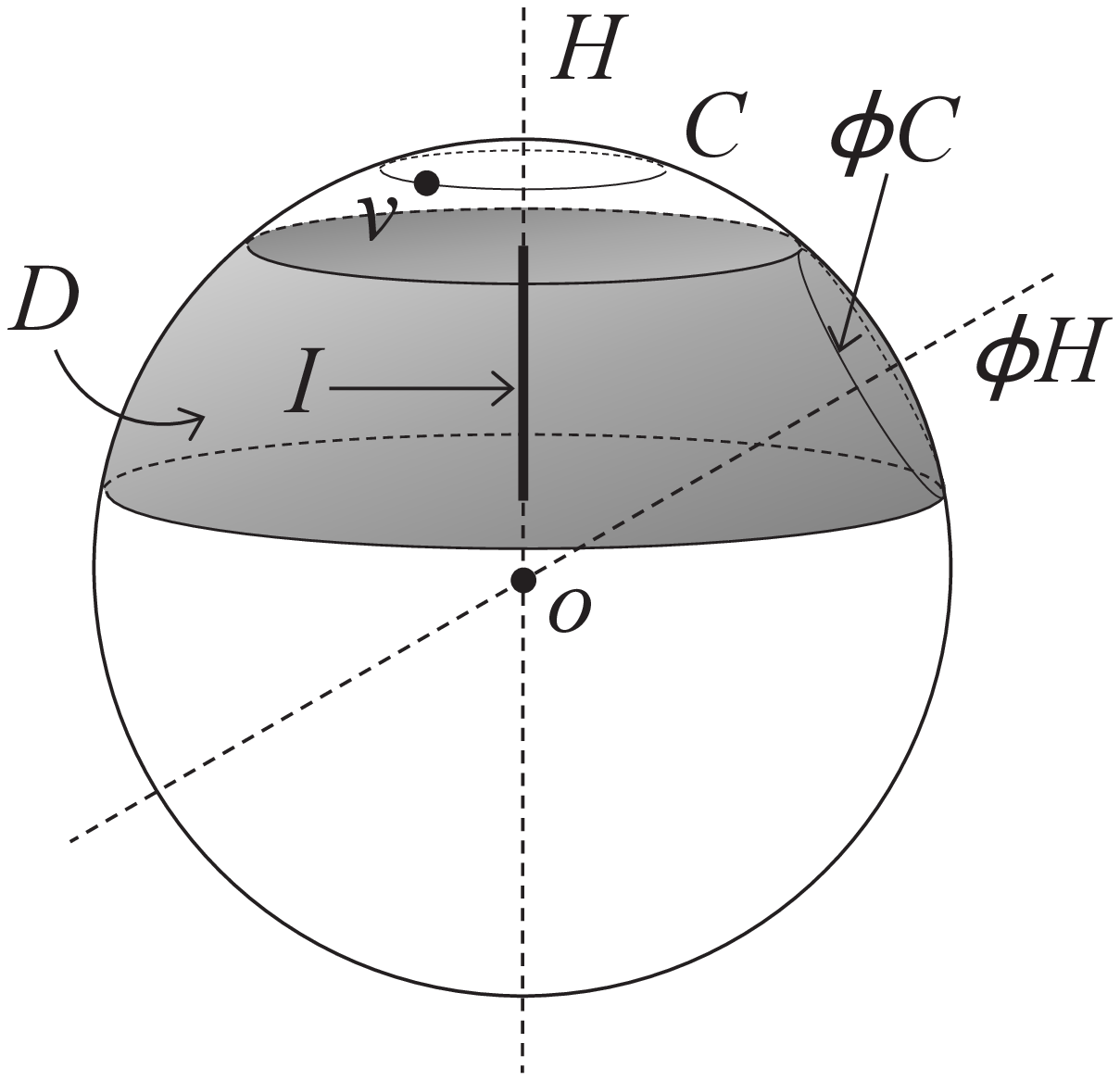, width=13.5cm}
\end{center}
\vspace{-0.5in}
\caption{}
\end{figure}

\begin{thm}\label{teo:sphericalsym_lines}
Let $H_j\in {\mathcal{G}}(n,1)$, $j=1,\dots,n$, be such that

\noindent{\rm{(i)}} at least two of them form an angle that is an irrational multiple of $\pi$,

\noindent{\rm{(ii)}} $H_1+\cdots+H_n=\R^n$, and

\noindent{\rm{(iii)}} $\{ H_1,\dots,H_n\}$ cannot be partitioned into two mutually orthogonal nonempty subsets.

If $E\subset\Sn$ is nonempty, closed, and such that $R_{H_j}E=E$, $j=1,\dots,n$, then $E=\Sn$.
\end{thm}

\begin{proof}
We claim that there is an ordering of $H_1,\dots, H_n$ so that for $m=2,\dots,n$ and all $v\in E$, we have
\begin{equation}\label{th_sphericalsym_lines_inductivehp}
\Sn\cap (H_1+\cdots+H_m+v)\subset E.
\end{equation}
When $m=n$, this yields the theorem, by hypothesis (ii).

We argue by induction on $m$.  By reordering, if necessary, we may, by hypothesis (i), assume that $H_1$ and $H_2$ form an angle that is an irrational multiple of $\pi$. The composition $R_{H_2}R_{H_1}$ maps the $2$-dimensional plane $H_1+H_2+v$ to itself. Moreover, $R_{H_2}R_{H_1}$ restricted to $H_1+H_2+v$ is a rotation by an irrational multiple of $\pi$.  By Kronecker's approximation theorem \cite[Theorem~7.7]{Apo}, the set $\{(R_{H_2}R_{H_1})^mv: m\in \N\}$ is dense in $\Sn\cap (H_1+H_2+v)$, so the invariance of $E$ with respect to $R_{H_2}R_{H_1}$ and its being closed imply that $\Sn\cap (H_1+H_2+v)\subset E$. This proves \eqref{th_sphericalsym_lines_inductivehp} when $m=2$.

Let $m\in\{2,\dots,n-1\}$ and assume that \eqref{th_sphericalsym_lines_inductivehp} holds for all $v\in E$. Define  $P_k=H_1+\cdots+H_k$ for $k=2,\dots,n$.  By hypothesis (iii), among the lines $H_{m+1},\dots,H_n$ there is at least one that is neither orthogonal to nor contained in $P_m$. By reordering, if necessary, we may assume that this line is $H_{m+1}$.

Let $v\in E$.  We want to prove that
\begin{equation}\label{theo_sph_sym_lines_eq1}
\Sn\cap(P_{m+1}+v)=\Sn\cap(P_m+H_{m+1}+v)\subset E.
\end{equation}
We apply Lemma~\ref{lemma_MS}, where $P_{m+1}+v$ plays the role of the ambient space $\R^n$ (so that $\Sn\cap(P_{m+1}+v)$ acts as the unit sphere), $E$ is replaced by $E\cap(P_{m+1}+v)$, $H$ is the line $P_m^\perp\cap(P_{m+1}+v)$, $O(n)_H$ is replaced by the isometries of $P_{m+1}+v$ to itself that fix $H$, and $\phi=R_{H_1}R_{H_{m+1}}$.  Note that $E\cap (P_{m+1}+v)$ is invariant with respect to $\phi$ since
$$R_{H_{m+1}}(E\cap (P_{m+1}+v))=E\cap (P_{m+1}-v)$$
and
$$R_{H_{1}}(E\cap (P_{m+1}-v))=E\cap (P_{m+1}+v).$$
The other hypotheses of Lemma~\ref{lemma_MS} are also satisfied. Indeed, $\dim(P_{m+1}+v)\ge 3$, and $\phi H\neq H$, because $H_{m+1}$ is neither orthogonal to nor contained in $P_m$; and the assumption that $E\cap (P_{m+1}+v)$ is invariant under each isometry of $P_{m+1}+v$ to itself that fixes $H$ comes from the inductive hypothesis \eqref{th_sphericalsym_lines_inductivehp}, which says that $S^{n-1}\cap (P_m+w)\subset E$ for each $w\in E$.  With all this in hand, Lemma~\ref{lemma_MS} gives \eqref{theo_sph_sym_lines_eq1}.
\end{proof}

Burchard, Chambers, and Dranovski \cite[Proposition~4.2]{BCD} prove a result similar to the previous theorem for reflections in hyperplanes through the origin. More precisely, they prove that if $H_j\in {\mathcal{G}}(n,n-1)$, $j=1,\dots,n$, then under the same assumptions on
$H_j^{\perp}$, $j=1,\dots,n$, as in Theorem~\ref{teo:sphericalsym_lines}, a nonempty closed subset of $\Sn$ that is symmetric with respect to $H_j$, $j=1,\dots,n$, coincides with $\Sn$.  In fact, in \cite[p.~1189]{BCD} they state without proof a stronger result, which is also a consequence of the main theorem of Eaton and Perlman \cite{EP}.  The latter states that if $G$ is a closed subgroup of $O(n)$ generated by reflections in $(n-1)$-dimensional subspaces that is infinite and irreducible (i.e., there is no nontrivial proper subspace $S$ of $\R^{n}$ such that $gS\subset S$ for $g\in G$), then $G=O(n)$.  Consequently, instead of condition (i) of Theorem~\ref{teo:sphericalsym_lines}, one need only assume that the subgroup generated by the reflections in $(n-1)$-dimensional subspaces is not a finite Coxeter subgroup of $O(n)$.  It appears to be unknown whether the result in \cite{EP} extends to reflections in lower-dimensional subspaces; in this connection, see Problem~\ref{probaug1}.

\begin{rem}\label{remoct16}
{\em The similarity between Theorem~\ref{teo:sphericalsym_lines} and \cite[Proposition~4.2]{BCD} suggests that there may be some duality at play.  To investigate this possibility, observe that if $H$ is a subspace of $\R^n$, then clearly
\begin{equation}\label{reflection_wrt_perp}
 R_{H^\perp}=-Id\circ R_{H}=R_H\circ-Id,
\end{equation}
where $Id$ is the identity map.  Now let $H_j\in {\mathcal{G}}(n,1)$, $j=1,\dots,n$, satisfy the hypotheses of Theorem~\ref{teo:sphericalsym_lines}.  Then $H_j^{\perp}\in {\mathcal{G}}(n,n-1)$, $j=1,\dots,n$, satisfy the hypotheses of \cite[Proposition~4.2]{BCD}.  Suppose that $E\subset S^{n-1}$ is nonempty, closed, and such that $R_{H_j}E=E$, $j=1,\dots,n$.  If $E$ is $o$-symmetric,  \eqref{reflection_wrt_perp} implies that $R_{H_j^{\perp}}E=E$, $j=1,\dots,n$, and then \cite[Proposition~4.2]{BCD} gives $E=S^{n-1}$.  Thus \cite[Proposition~4.2]{BCD} implies Theorem~\ref{teo:sphericalsym_lines} when $E$ is $o$-symmetric, and a similar argument shows that Theorem~\ref{teo:sphericalsym_lines} implies \cite[Proposition~4.2]{BCD} when $E$ is $o$-symmetric. For general sets $E\subset S^{n-1}$, however, the relationship between these results is not obvious; see Problem~\ref{proboct16}.}
\end{rem}

To obtain results similar to Theorem~\ref{teo:sphericalsym_lines} for the case when $i\in \{2,\dots,n-2\}$, some preliminary observations will be useful.  Let $H_1$ and $H_2$ be subspaces of $\R^n$ of dimension  $k$ and $i$, respectively, with $k\geq i$. There exists a unique increasing sequence $\alpha_1,\dots,\alpha_i$ in $[0,\pi/2]$ and an orthonormal basis $e_1,\dots,e_n$ for $\R^n$ such that
\begin{eqnarray}\label{angles}
H_1&=&\lin\{ e_1,\dots,e_k\}\quad\text{and }\nonumber\\
H_2&=&\lin\{\cos \alpha_j\, e_j+\sin \alpha_j\, e_{i+k+1-j} : j=1,\dots,i\}\nonumber\\
&=&\lin\{\cos \alpha_1\, e_1+\sin \alpha_1\, e_{k+i},\dots,\cos \alpha_i\, e_i+\sin \alpha_i\, e_{k+1}\},
\end{eqnarray}
where the first $i+k-n$ angles are zero if $i+k>n$.  Then it is easy to see that any $x\in\R^n$ can be written as
\begin{equation}\label{representation_x}
x=\sum_{j=1}^i\rho_j(\cos\theta_j\, e_j+\sin \theta_j\, e_{i+k+1-j}) +y_1+y_2
\end{equation}
for suitable $\rho_1,\dots,\rho_i\in\R$, $\theta_1,\dots,\theta_i\in[0,2\pi)$, $y_1\in\lin\{e_{i+1},\dots,e_k\}$ (where $y_1=o$ when $i=k$) and $y_2\in(H_1+H_2)^\perp$.

\begin{lem}\label{pr_two_i_reflections}
Let the subspaces $H_1$ and $H_2$, of dimension $k$ and $i\le k$, respectively, be as in \eqref{angles} and let the point $x\in\R^n$ be specified by \eqref{representation_x}. Then
\begin{equation}\label{representation_sh1}
R_{H_1} x=\sum_{j=1}^i\rho_j(\cos(-\theta_j)\, e_j+\sin (-\theta_j)\, e_{i+k+1-j})+y_1-y_2,
\end{equation}
\begin{equation}\label{representation_sh2}
R_{H_2} x=\sum_{j=1}^i\rho_j(\cos(2\alpha_j-\theta_j)\, e_j+\sin (2\alpha_j-\theta_j)\, e_{i+k+1-j}) -y_1-y_2,
\end{equation}
and for $m\in\N$,
\begin{equation}\label{representation_sh2sh1}
(R_{H_2}R_{H_1})^m x=\sum_{j=1}^i\rho_j(\cos(2m\alpha_j+\theta_j)\, e_j+\sin (2m\alpha_j+\theta_j)\, e_{i+k+1-j}) +(-1)^my_1+y_2.
\end{equation}
\end{lem}

\begin{proof}
To obtain \eqref{representation_sh1}, note that $R_{H_1}$ keeps the component of $x$ in $\lin\{e_1,\dots,e_k\}$ unchanged and changes the sign of the component in $\lin\{e_{k+1},\dots,e_n\}$.

To prove~\eqref{representation_sh2}, we first compute $x|H_2$. Since the vectors in the representation of $H_2$ are orthonormal, we have
$$
x|H_2=\sum_{j=1}^i\big(x\cdot(\cos \alpha_j\, e_j+\sin \alpha_j\, e_{i+k+1-j})\big)\big(\cos \alpha_j\, e_j+\sin \alpha_j\, e_{i+k+1-j}\big).
$$
If  $x$ is as in \eqref{representation_x}, then
$$
x\cdot(\cos \alpha_j\, e_j+\sin \alpha_j\, e_{i+k+1-j})=\rho_j\cos(\alpha_j-\theta_j)
$$
and therefore
$$
x|H_2=\sum_{j=1}^i\rho_j\cos(\alpha_j-\theta_j)(\cos\alpha_j\, e_j+\sin\alpha_j\, e_{i+k+1-j}).
$$
Since $R_{H_2} x=2(x|H_2)-x$, we obtain
\begin{multline*}
R_{H_2} x=\sum_{j=1}^i\rho_j(2\cos(\alpha_j-\theta_j)\cos\alpha_j - \cos\theta_j) e_j+\\
+\sum_{j=1}^i\rho_j(2\cos(\alpha_j-\theta_j)\sin\alpha_j - \sin\theta_j) e_{i+k+1-j}-y_1-y_2.
\end{multline*}
Then the relations
$$
2\cos(\alpha_j-\theta_j)\cos\alpha_j - \cos\theta_j=\cos(2\alpha_j-\theta_j)$$
and
$$2\cos(\alpha_j-\theta_j)\sin\alpha_j - \sin\theta_j=\sin(2\alpha_j-\theta_j)$$
yield \eqref{representation_sh2}.  Finally,  \eqref{representation_sh2sh1} is an immediate consequence of \eqref{representation_sh1} and \eqref{representation_sh2}.
\end{proof}

It follows from \eqref{reflection_wrt_perp} that if $H_1$ and $H_2$ are subspaces of $\R^n$, then
\begin{equation}\label{oct15eq}
 R_{H_2^\perp}R_{H_1^\perp}=R_{H_2}R_{H_1}.
\end{equation}

Recall that if $1\le p\neq q\le n$, then $SO(n)_{{\lin\{e_p,e_{q}\}}^{\perp}}$ denotes the isometries in $SO(n)$ that act as the identity on $\lin\{e_p,e_{q}\}^{\perp}$, the $(n-2)$-dimensional subspace orthogonal to $e_p$ and $e_{q}$.

\begin{lem}\label{pr_kronecker}
Let $i\leq n/2$ and let $\alpha_1,\dots,\alpha_i$ be an increasing sequence in $(0,\pi/2)$ such that $\pi, \alpha_1,\dots, \alpha_i$ are linearly independent over $\Q$. Let $H_1, H_2\in {\mathcal{G}}(n,i)$ have representations as in \eqref{angles}.
Then both
\begin{equation}\label{kronecker}
\{(R_{H_2}R_{H_1})^m:\ m\in\N\}~~\quad{\text{and}}~~\quad \{(R_{H_2^\perp}R_{H_1^\perp})^m:\ m\in\N\}
\end{equation}
are dense in
$SO(n)_{{\lin\{e_j,\,e_{2i-j+1}\}}^{\perp}}$, $j=1,\dots,i$.
\end{lem}

\begin{proof}
Let $j\in \{1,\dots,i\}$. In (\ref{representation_x}), let $x=e_j$, so that $y_1=y_2=o$, $\rho_j=1$, $\rho_l=0$ for $l\neq j$, and $\theta_j=0$.  With $k=i$, (\ref{representation_sh2sh1}) yields
\begin{equation}\label{sep20eq}
(R_{H_2}R_{H_1})^m e_j=\cos 2m\alpha_j\, e_j+\sin 2m\alpha_j\, e_{2i-j+1}.
\end{equation}
The claim regarding the set on the left in \eqref{kronecker} is then a direct consequence of the linear independence of $\pi, \alpha_1,\dots, \alpha_i$ over $\Q$ and Kronecker's approximation theorem \cite[Theorem~7.7]{Apo}.  Indeed, the latter implies that because $\alpha_j$ is not a rational multiple of $\pi$, the set $\{(\cos 2m\alpha_j, \sin 2m\alpha_j): m\in \N\}$ is dense in $S^1$.  By \eqref{sep20eq}, it follows that $\{(R_{H_2}R_{H_1})^m e_j:m\in\N\}$ is dense in $S^{n-1}\cap \lin\{e_j,e_{2i-j+1}\}$ and hence, since $\dim(\lin\{e_j,e_{2i-j+1}\})=2$, that $\{(R_{H_2}R_{H_1})^m: m\in\N\}$ is
dense in $SO(n)_{{\lin\{e_j,\,e_{2i-j+1}\}}^{\perp}}$. This also proves the claim regarding the set on the right in \eqref{kronecker}, since by \eqref{oct15eq}, the two sets in \eqref{kronecker} are equal.
\end{proof}

\begin{lem}\label{pr_sh1sh2sh3}
Let $2\leq i\leq n/2$ and let $\alpha_1,\dots,\alpha_i$ be an increasing sequence in $(0,\pi/2)$ such that $\pi, \alpha_1,\dots, \alpha_i$ are linearly independent over $\Q$. Let $H_1, H_2, H_3\in{\mathcal{G}}(n,i)$ be defined by
\begin{equation}\label{choice_of_the_planes}
\begin{aligned}
H_1=&\lin\{ e_1,e_3,\dots,e_{2i-1}\},\\
H_2=&\lin\{\cos \alpha_j\, e_{2j-1}+\sin \alpha_j\, e_{2j} : j=1,\dots,i\},\quad{\text{and}}\\
H_3=&\lin\{\{\cos \alpha_1\, e_{1}+\sin \alpha_1\, e_{2i}\}\cup\{\cos \alpha_j\, e_{2j-1}+\sin \alpha_j\, e_{2j-2}\} : j=2,\dots,i\}.
\end{aligned}
\end{equation}
Let $E\subset S^{n-1}$ be nonempty, closed, and such that either $R_{H_j}E=E$, $j=1,2,3$, or $R_{H_j^\perp}E=E$, $j=1,2,3$. Then if $x\in E$, we have $\Sn\cap\left(\lin\{e_1,\dots,e_{2i}\}+x\right)\subset E$.
\end{lem}

\begin{proof}
Assume that $R_{H_j}E=E$ for $j=1,2,3$. Since $E$ is closed, it is invariant with respect to each element of
$\cl\{(R_{H_2}R_{H_1})^m:\ m\in\N\}\cup\cl \{(R_{H_3}R_{H_1})^m:\ m\in\N\}$. The same conclusion holds if we assume instead that $R_{H_j^\perp}E=E$ for $j=1,2,3$, since by \eqref{oct15eq}, $R_{H_2}R_{H_1}=R_{H_2^\perp}R_{H_1^\perp}$ and $R_{H_3}R_{H_1}=R_{H_3^\perp}R_{H_1^\perp}$.  By Lemma~\ref{pr_kronecker}, modified suitably in view of the different expressions for $H_1$ and $H_2$ in (\ref{angles}) with $k=i$ and in (\ref{choice_of_the_planes}), $E$ is invariant with respect to each element of
$$
V_1=\cup\{SO(n)_{\lin\{e_{2j-1},\,e_{2j}\}^{\perp}}: j=1,\dots,i\}.
$$
By using Lemma~\ref{pr_kronecker} similarly but with $H_2$ replaced by $H_3$, we see that $E$ is also invariant with respect to each element of
$$
V_2=\cup\{SO(n)_{\lin\{e_{2j-1},\,e_{2j-2}\}^{\perp}}: j=2,\dots,i\}.
$$
(The factor $SO(n)_{\lin\{e_{1},\,e_{2i}\}^{\perp}}$ could also be included in the latter union, but we do not require this for the remaining argument.) Consequently, $E$ is invariant with respect to each element of the subgroup $G$ of $SO(n)$ generated by $V_1\cup V_2$.

Suppose that $x\in E$ and choose $y\in \lin\{e_1,\dots,e_{2i}\}^\perp$ such that
$$S^{n-1}\cap\left(\lin\{e_1,\dots,e_{2i}\}+y\right)=
S^{n-1}\cap\left(\lin\{e_1,\dots,e_{2i}\}+x\right).$$
It clearly suffices to prove that $G$ acts transitively on $S^{n-1}\cap\left(\lin\{e_1,\dots,e_{2i}\}+y\right)$.  To see this, let $u\in S^{n-1}\cap\left(\lin\{e_1,\dots,e_{2i}\}+y\right)$ and choose
$z=(z_1,\dots,z_{2i})\in\lin\{e_1,\dots,e_{2i}\}$ such that $u=y+z$. Then there is a $\phi_1\in SO(n)_{\lin\{e_{2i-1},\,e_{2i}\}^{\perp}}\subset V_1$ such that
$$\phi_1 u=\phi_1 y+\phi_1 z=y+\left(z_1,z_2,\dots,z_{2i-2},\sqrt{z_{2i-1}^2+z_{2i}^2},0\right)$$
and hence a $\phi_2\in  SO(n)_{\lin\{e_{2i-1},\,e_{2i-2}\}^{\perp}}\subset V_2$ such that
$$\phi_2 \phi_1 u=\phi_2 y+\phi_2 \phi_1 z=y+\left(z_1,z_2,\dots,z_{2i-3},\sqrt{z_{2i-2}^2+z_{2i-1}^2+
z_{2i}^2},0,0\right).$$
Observe that the graph with vertices $\{1,2,\dots,2i\}$ and edges joining ${2j-1}$ to ${2j}$, $j=1,\dots,i$ (corresponding to $V_1$), and joining ${2j-1}$ to ${2j-2}$, $j=2,\dots,i$ (corresponding to $V_2$), is connected, since each vertex is joined to the preceding one and the following one. Using this property, we may choose $\phi_3,\dots,\phi_{2i-1}\in V_1\cup V_2\subset G$, in the same way that we chose $\phi_1$ and $\phi_2$, so that
$$\phi_{2i-1}\cdots \phi_1 u=y+\left(\|z\|,0,\dots,0\right)=y+\left(\sqrt{1-\|y\|^2},0,\dots,0\right).$$
Now if $u'\in S^{n-1}\cap\left(\lin\{e_1,\dots,e_{2i}\}+y\right)$, we can similarly find $\psi_3,\dots,\psi_{2i-1}\in G$ such that
$$\psi_{2i-1}\cdots \psi_1 u'=y+\left(\sqrt{1-\|y\|^2},0,\dots,0\right)=\phi_{2i-1}\cdots \phi_1 u.$$
Thus there is a $g\in G$ with $gu'=u$, so $G$ acts transitively on $S^{n-1}\cap\left(\lin\{e_1,\dots,e_{2i}\}+y\right)$.
\end{proof}

The next lemma says that if $H$ and $L$ are subspaces satisfying the stated hypotheses and a nonempty closed set $E\subset\Sn$ is $H$-symmetric and rotationally symmetric with respect to $L^\perp$, then $E$ is also rotationally symmetric with respect to $(H+L)^\perp$.  A couple of small modifications to the proof show that the lemma also holds when instead of $\dim H < \dim L$ it is assumed that $\dim H=\dim L$ and $H\cap L\neq\{o\}$, but we have no use for this case in the sequel.

\begin{lem}\label{extension_lemma} {\em (Symmetry extension lemma.)}
Let $H$ and $L$ be subspaces of $\R^n$ such that $H \cap L^\perp=\{o\}$ and $\dim H < \dim L$.  If $E\subset \Sn$ is nonempty, closed, and invariant under $R_{H}$ and under any element of $O(n)_{L^\perp}$, then $E$ is invariant under any element of $O(n)_{(H+L)^\perp}$. The same conclusion holds if invariance under $R_{H}$ is replaced by invariance under $R_{H^\perp}$.
\end{lem}

\begin{proof}
Let $G$ be the closure of the subgroup of $O(n)$ generated by $O(n)_{L^\perp}$ and $R_H$. The first conclusion of the lemma is equivalent to the statement that for all $x\in E$, the orbit $G(x)$ contains the orbit $O(n)_{(H+L)^\perp}(x)$. The latter orbit is  $\Sn\cap(H+L+x)$, and we claim that it coincides with $S_x=\Sn\cap (M+x)$, where $M=L+R_HL$.  (See Figure~2, where for the sake of illustration, both $H$ and $L$ are 1-dimensional.)  It suffices to prove that $M=H+L$. To this end, we use the easily-proved identity $A^{\perp}\cap B^{\perp}=(A+B)^{\perp}$, where $A$ and $B$ are subspaces of $\R^n$, to conclude that the assumption $H\cap L^\perp =\{o\}$ is equivalent to $H^\perp + L=\R^n$. Then $L|H=(H^\perp + L)|H=\R^n|H=H$.  Hence, if $z\in H+L$, then $z=x|H+y$ for some $x,y\in L$. Thus
$$z=(x/2+y)+(x|H-x/2)=(x/2+y)+R_H(x/2)\in L+R_HL.$$
This shows that $H+L\subset M$. On the other hand, the definition of $R_H$ implies that $R_HL\subset (L|H)+L\subset H+L$ and hence that $M\subset H+L$.  This proves the claim.

\begin{figure}[htb]\label{fig2}
\begin{center}
%\vspace*{1cm}
\epsfig{file=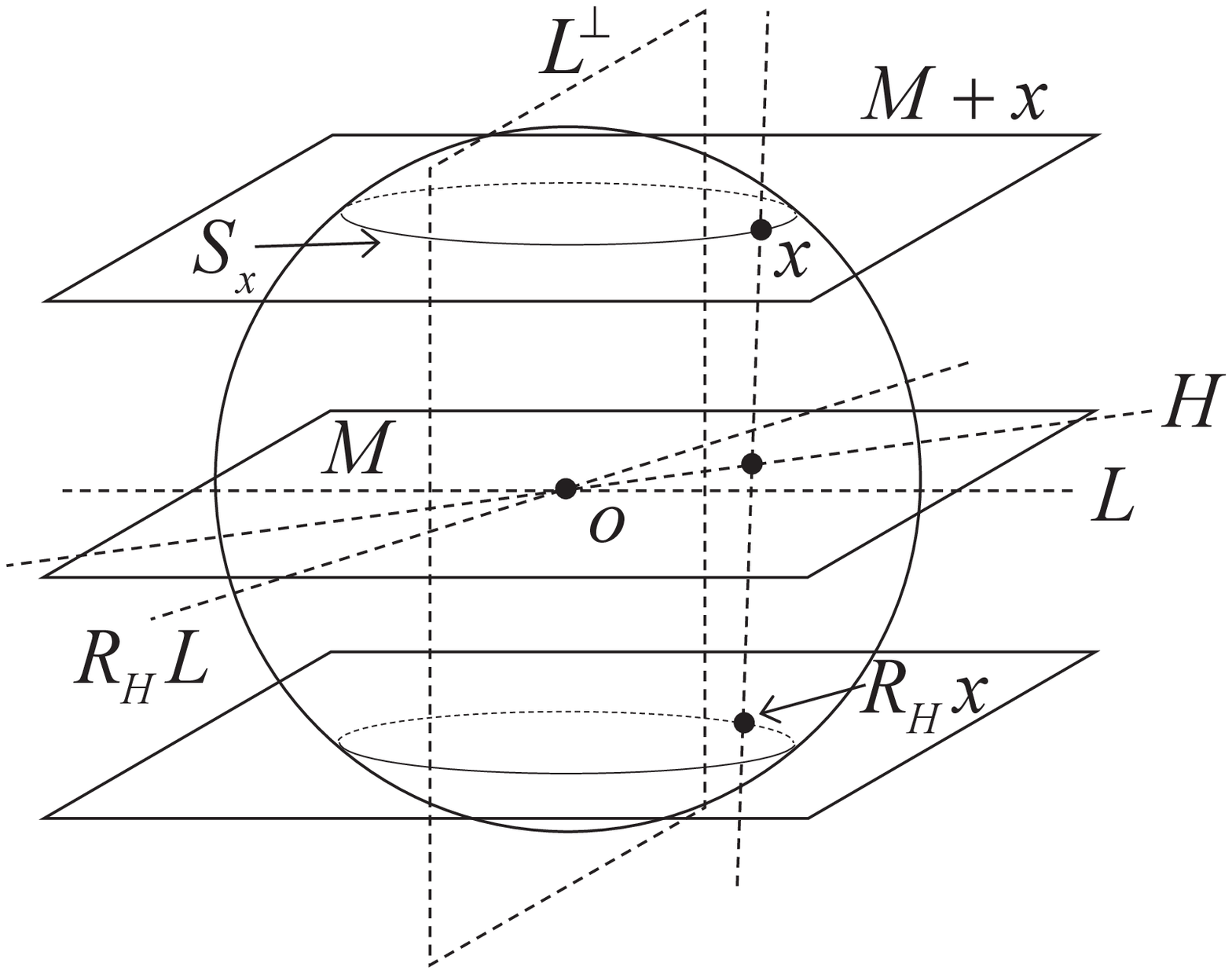, width=13.5cm}
\end{center}
\vspace{-0.3in}
\caption{}
\end{figure}

Let $x\in E$.  We have to show that $S_x\subset G(x)$. If $x\in M^\perp$, then $S_x=\{ x\}$ and there is nothing to prove. Therefore we may assume that $x\notin M^\perp$, in which case $S_x$ is of dimension $\dim M-1$. We may also assume that $\dim H>0$, since otherwise $H+L=L$ and the first conclusion of the lemma is trivially true.  Then $\dim L>1$, which in turn yields $\dim M=\dim\,(H+L)>1$, so $S_x$ is connected.  In order to prove that $G(x)$ contains $S_x$, we shall prove that $G(x)\cap S_x$ is both closed and open relative to $S_x$.

To prove that $G(x)\cap S_x$ is closed relative to $S_x$, let $\{y_k\}$ be a sequence of points in $G(x)\cap S_x$ converging to $y$. By the definition of $G(x)$, for each $k$, there is a  $\phi_k\in O(n)$ such that $\phi_k(x)=y_k$. Since $O(n)$ is compact, there exists a subsequence $\{\phi_{k_m}\}$ of $\{\phi_{k}\}$ converging to $\phi\in O(n)$. Clearly $\phi(x)=y$, so $y\in G(x)$, and $y\in S_x$ because $S_x$ is closed.

It remains to show that $G(x)\cap S_x$ is open relative to $S_x$.   By the closed-subgroup theorem (see, for example, \cite[Theorem~3.42]{War}), $G$ is a Lie group, because it is a closed subgroup of a Lie group.  Since $G$ is a $C^\infty$ manifold and the map $a:G\to G(x)$ given by $a(g)=g(x)$ is $C^\infty$, the orbit $G(x)$ is also a $C^\infty$ manifold.  Moreover, $G(x)\cap S_x$ is also a manifold, because $G(x)\subset S_x\cup R_H S_x=S_x\cup  S_{R_H x}$, and the manifolds $S_x$ and $S_{R_H x}$ either coincide or are disjoint.  Since $G(x)\cap S_x$ is a submanifold of $S_x$, it is open relative to $S_x$ if (and only if) the dimension of $G(x)\cap S_x$ equals the dimension $\dim M-1$ of $S_x$.

Let $p\in G(x)\cap S_x$. The orbit of $p$ with respect to each closed subgroup of $G$ is a submanifold of $G(x)$. Two such subgroups are natural for our purposes: $O(n)_{L^\perp}$ and its conjugate $R_HO(n)_{L^\perp}R_H$. The orbits of $p$ with respect to them are $\Sn \cap (L+p)$ and $\Sn \cap (R_HL+p)$, respectively, both of which are contained in $G(x)\cap S_x$.  Denote by $T_pG(x)$ and $T_pS_x$ the tangent spaces to $G(x)$ and $S_x$, respectively, at $p$. We know that $T_pS_x = M\cap p^\perp$.  Also, $T_pG(x)$ contains both $L\cap p^\perp$ and $R_HL\cap p^\perp$, the tangent spaces to the submanifolds $\Sn \cap (L+p)$ and $\Sn \cap (R_HL+p)$ just mentioned. The tangent space to a $C^\infty$ manifold at any point is a vector space whose dimension equals that of the manifold; see, for example, \cite[p.~12]{War} and \cite[Corollary to Theorem~1.17]{War}. With these facts at hand, it is enough to show that we can choose $p$ so that
\begin{equation}\label{dimcond0}
\dim \left( (L\cap p^\perp) + (R_HL\cap p^\perp) \right) \geq \dim M -1.
\end{equation}
Now
$$
\dim \left( (L\cap p^\perp) + (R_HL\cap p^\perp) \right) = \dim (L\cap p^\perp) + \dim (R_HL\cap p^\perp) -\dim (L\cap R_HL\cap p^\perp)
$$
and
$$\dim M=\dim L + \dim R_HL -\dim (L\cap R_HL).$$
The previous two equations show that \eqref{dimcond0} is equivalent to
\begin{equation}\label{dimcond}
\dim L-\dim (L\cap p^\perp)+\dim R_HL-\dim (R_HL\cap p^\perp) \leq
\dim (L\cap R_HL)-\dim (L\cap R_HL\cap p^\perp)+1.
\end{equation}
Next, we observe that the following equations hold:
\begin{eqnarray*}
 \dim L-\dim (L\cap p^\perp)&=&\left\{ \begin{array}{ll}0,\ \  {\text{if}} \ p\in L^\perp \\ 1, \ \ {\text{if}}\  p\notin L^\perp, \end{array}\right. \\
 \dim R_HL-\dim (R_HL\cap p^\perp)&=&\left\{ \begin{array}{ll}0,\ \  {\text{if}} \ p\in R_HL^\perp \\ 1, \ \ {\text{if}}\  p\notin R_HL^\perp, \end{array}\right. \\
 \dim (L\cap R_HL)-\dim (L\cap R_HL\cap p^\perp)&=&\left\{ \begin{array}{ll}0,\ \  {\text{if}} \ L\cap R_HL\subset p^\perp \\ 1, \ \ {\text{if}}\  L\cap R_HL\not\subset p^\perp.\end{array}\right. \\
\end{eqnarray*}
It follows that \eqref{dimcond} holds for any $p$ such that $L\cap R_HL\not\subset p^\perp$, or equivalently, for any $p\not\in L^\perp+R_H L^\perp$.

We have $x\notin M^\perp=L^\perp\cap R_H L^\perp$, and thus either $x\notin L^\perp$ or $x\notin R_H L^\perp$. We claim that if $x\notin L^\perp$, then $\Sn\cap(L+x)$ contains such a $p$, while if $x\notin R_H L^\perp$, then $\Sn\cap(R_H L+x)$ contains such a $p$. It is enough to deal with the first case, since the proof for the second one is similar.
Thus we need to prove that if $x\notin L^\perp$ then
\begin{equation}\label{find_p}
 \Sn\cap (L+x)\not\subset L^\perp+R_H L^\perp.
\end{equation}
Suppose that \eqref{find_p} is false. Then $L^\perp+R_H L^\perp$ contains all vectors which are differences of points in $\Sn\cap (L+x)$. The assumption $x\notin L^\perp$ implies that $\Sn\cap (L+x)$ has dimension $\dim L-1>0$. It follows that $L^\perp+R_H L^\perp$ contains $L$. Therefore $L^\perp+R_H L^\perp=\R^n$, since $L^\perp+R_H L^\perp$ contains both $L$ and $L^\perp$.  From $H \cap L^\perp=\{o\}$ and $\dim H < \dim L$, we obtain
$$\dim(H^\perp\cap L)=\dim H^\perp+\dim L-\dim(H^\perp+L)
=n-\dim H+\dim L-n>0,$$
so there exists $w\neq o$ in $H^\perp\cap L$. Clearly, $R_Hw=-w$ and $w\in L\cap R_HL$. Thus $L\cap R_HL\neq\{o\}$ and $L^\perp+R_H L^\perp=(L\cap R_HL)^\perp\neq\R^n$. This contradiction proves \eqref{find_p}.

When $R_{H^\perp} E=E$, the proof is substantially identical. Due to \eqref{reflection_wrt_perp}, we have $R_HL=R_{H^{\perp}}L$, allowing us to repeat the arguments above with $R_H$ replaced by $R_{H^\perp}$.
\end{proof}

The next result provides for $i\in \{2,\dots,n-2\}$ a possible choice of $i$-dimensional subspaces $H_1,\dots, H_k$ such that the reflections $R_{H_1},\dots,R_{H_k}$ generate $O(n)$.

\begin{thm}\label{finthm}
Let $k\ge 3$ and let $H_j\in{\mathcal{G}}(n,i)$, $j=1,\dots,k$. If $2\le i\le n/2$, assume that

{\noindent{\rm{(i)}}} $H_1$, $H_2$, and $H_3$ are as in \eqref{choice_of_the_planes},

{\noindent{\rm{(ii)}}}  $H_1+\dots+H_k=\R^n$, and

{\noindent{\rm{(iii)}}} for each $j=3,\dots,k-1$,
$$
H_{j+1}\cap\left(H_1+\dots+H_{j}\right)^\perp=\{o\}.
$$
If $n/2< i\leq n-2$, assume that {\em (i)--(iii)} are satisfied with each $H_j$ replaced by $H_j^\perp$.  If $E\subset \Sn$ is nonempty, closed, and such that $R_{H_j} E=E$ for $j=1,\dots,k$, then $E=\Sn$.
\end{thm}

\begin{proof}
Suppose that $2\leq i\leq n/2$. Note that $H_1+H_2+H_3=\lin\{e_1,\dots,e_{2i}\}$.  Therefore the assumptions on $E$ and hypothesis (i), together with Lemma~\ref{pr_sh1sh2sh3}, imply that $E$ is invariant under any element of $O(n)_{(H_1+H_2+H_3)^\perp}$.  Let $3\le j\le k-1$ and suppose that $E$ is invariant under any element of $O(n)_{(H_1+\dots+H_j)^\perp}$.  Since
$$
\dim H_{j+1}=i<2i=\dim(H_1+H_2+H_3)\leq\dim(H_1+\dots+H_{j}),
$$
we may by hypothesis (iii) apply Lemma~\ref{extension_lemma} with $H$ and $L$ replaced by $H_{j+1}$ and $H_1+\dots+H_{j}$, respectively, to conclude that $E$ is invariant under any element of $O(n)_{(H_1+\dots+H_{j+1})^\perp}$.  In view of hypothesis (ii), the proof is completed by induction on $j$. When $n/2< i\leq n-2$, the previous argument can be repeated with each $H_j$ replaced by $H_j^\perp$.
\end{proof}

\begin{cor}\label{cor21June}
Let $1\leq i\leq n-1$ and let
$$
k=
\begin{cases}
n, & \text{if $i=1$ or $i=n-1$,}\\
\lceil n/i\rceil+1, & \text{if $1<i\leq n/2$,}\\
\lceil n/(n-i)\rceil+1, & \text{if $n/2\leq i<n-1$.}
\end{cases}
$$
There exist $H_j\in{\mathcal{G}}(n,i)$, $j=1,\dots,k$, such that if $E\subset \Sn$ is nonempty, closed, and such that $R_{H_j} E=E$ for $j=1,\dots, k$, then $E=\Sn$.  Hence, if $F\subset \R^n$ is closed and $R_{H_j} F=F$ for $j=1,\dots, k$ (or $K\in{\mathcal{K}}^n_n$ and $R_{H_j}K=K$ for $j=1,\dots, k$), then $F$ is a union of $o$-symmetric spheres (or $K$ is an $o$-symmetric ball, respectively).
\end{cor}

\begin{proof}
Let $E\subset \Sn$ satisfy the stated hypotheses.  If $i=1$ or $i=n-1$, $E=\Sn$ is a consequence of Theorem~\ref{teo:sphericalsym_lines} or of \cite[Proposition~4.2]{BCD}, respectively. Suppose that $1<i\leq n/2$ and $k=\lceil n/i\rceil+1$.  Then if $H_j\in{\mathcal{G}}(n,i)$, $j=1,\dots,k$, are as in Theorem~\ref{finthm}, hypothesis (iii) there and $\dim(H_1+H_2+H_3)=2i$ imply that
$$\dim(H_1+\cdots+H_k)=2i+(k-3)i\ge 2i+(n/i-2)i=n$$
and hence hypothesis (ii), so it follows from that theorem that $E=\Sn$.  Suppose that $n/2\leq i<n-1$ and $k\ge\lceil n/(n-i)\rceil+1$. Choose $H_j\in{\mathcal{G}}(n,i)$, $j=1,\dots,k$, so that $H_j^{\perp}\in{\mathcal{G}}(n,n-i)$, $j=1,\dots,k$, are as in Theorem~\ref{finthm}.  Then $E=\Sn$ follows from the already established result for $1<i\leq n/2$ on replacing $i$ by $n-i$.

Let $F\subset \R^n$ be closed and satisfy $R_{H_j} F=F$ for $j=1,\dots, k$.  If $F$ is not a union of $o$-symmetric spheres, there is an $r>0$ such that if $A=F\cap rS^{n-1}$, then $A$ is a nonempty, closed, proper subset of $rS^{n-1}$.  It follows that the set $E=(1/r)A\subset S^{n-1}$ satisfies $R_{H_j} E=E$ for $j=1,\dots, k$, but $E\neq S^{n-1}$, contradicting what has been proved.  The proof is completed by noting that if $K\in{\mathcal{K}}^n_n$ is a union of $o$-symmetric spheres, it must be an $o$-symmetric ball.
\end{proof}

\section{Full rotational symmetry via rotational symmetries in finitely many subspaces}\label{rotationalsymmetry}

This section focuses on finding finite sets of $i$-dimensional subspaces such that (full) rotational symmetries in these subspaces generate full rotational symmetry.

\begin{thm}\label{new}
Let $H_1,\dots,H_k$ be subspaces in $\R^n$ such that $1\le \dim H_j\le n-2$ for $j=1,\dots,k$.  The following statements are equivalent.

\noindent{\rm{(i)}} $H_1^{\perp}+\cdots+H_k^{\perp}=\R^n$ and $\{H_1^{\perp},\dots,H_k^{\perp}\}$ cannot be partitioned into two mutually orthogonal nonempty subsets.

\noindent{\rm{(ii)}} If $E\subset S^{n-1}$ is nonempty and closed, and for each $j=1,\dots,k$ and $x\in E$, we have $\Sn\cap (H_j^{\perp}+x)\subset E$, then $E=S^{n-1}$.

\noindent{\rm{(iii)}}
If $F\subset \R^n$ is closed and invariant under any rotation that fixes $H_j$ for each $j=1,\dots,k$, then $F$ is a union of $o$-symmetric spheres.

\noindent{\rm{(iv)}}
If $K\in{\mathcal{K}}^n_n$ is rotationally symmetric with respect to $H_j$ for each $j=1,\dots,k$, then $K$ is an $o$-symmetric ball.
\end{thm}

\begin{proof}
(i)$\Rightarrow$ (ii) Suppose that (i) holds. If $E$ satisfies the hypotheses in (ii), then every section of $E$ with a translate of $H_1^{\perp}$ is either empty or a sphere. By (i), there must be a $j_0\in \{2,\dots,k\}$ such that $H_{j_0}^{\perp}$ is not contained in $H_1$, for otherwise $\{H_1\}$ and $\{H_2,\dots,H_k\}$ would be a partition of $\{H_1,\dots,H_k\}$ rendering (i) false.  Without loss of generality, assume that $j_0=2$ and choose an orthogonal basis $v_1,\dots,v_l$ in $H_2^{\perp}$ such that $v_1\not\in H_1$, i.e., $v_1$ is not orthogonal to $H_1^{\perp}$. If $x\in S^{n-1}$, then
$$R_{v_1^\perp} x \in \lin \{v_1\}+x\subset H_2^\perp +x.$$
The hypothesis on $E$ in (ii) with $j=2$ thus implies that if $x \in E$, then $R_{v_1^\perp}x \in E$, that is, that $E$ is symmetric with respect to the hyperplane $v_1^\perp$.  By Lemma~\ref{extension_lemma} with $H=\lin\{v_1\}$ and $L=H_1^{\perp}$, every section of $E$ with a translate of $H_1^{\perp}+\lin\{v_1\}$ is either empty or a sphere.  For $p=1,\dots,l$, define $w_p=v_1+\cdots+v_p$.  It is easy to see that $w_p$ is not orthogonal to $w_1,\dots,w_{p-1}$. Let $F_p=\lin\{w_1,\dots,w_p\}$, $p=1,\dots,l$.  Using Lemma~\ref{extension_lemma} iteratively with $H=\lin\{w_p\}$ and $L=H_1^{\perp}+F_{p-1}$, $p=1,\dots,l$, we conclude that every section of $E$ with a translate of $H_1^{\perp}+F_l=H_1^{\perp}+H_2^{\perp}$ is either empty or a sphere.

Suppose we have shown that every section of $E$ with a translate of $J_r=H_1^{\perp}+\cdots+H_r^{\perp}$ is either empty or a sphere, where $2\le r\le k-1$.  By (i), there must be a $j_1\in \{r+1,\dots,k\}$ such that $H_{j_1}^{\perp}$ is not contained in $J_r^{\perp}$, for otherwise (i) would be false via the partition $\{H_1,\dots,H_r\}$ and $\{H_{r+1},\dots,H_k\}$.  Without loss of generality, assume that $j_1=r+1$ and apply the argument in the previous paragraph, with $H_1^{\perp}$ and $H_2^{\perp}$ replaced by $J_r$ and $H_{r+1}^{\perp}$, respectively.  This yields that every section of $E$ with a translate of $J_{r+1}$ is either empty or a sphere.  By induction on $r$, we arrive at the conclusion that every section of $E$ with a translate of $J_{k}$ is either empty or a sphere.  But $J_k=H_1^{\perp}+\cdots+H_k^{\perp}=\R^n$, so $E=S^{n-1}$.

(ii)$\Rightarrow$ (iii)$\Rightarrow$ (iv) These implications are proved by the arguments in the last paragraph of the proof of Corollary~\ref{cor21June}.

(iv)$\Rightarrow$ (i)  Suppose that (i) is false and $H_1^{\perp}+\cdots+H_k^{\perp}=\R^n$.  Then there is a partition of $\{1,\dots,k\}$ into two nonempty disjoint sets $I_1$ and $I_2$ such that the subspaces $S_m=\sum\{H_j^{\perp}:j\in I_m\}$, $m=1,2$, are orthogonal.  Let $K=(B^n\cap S_1)\times (B^n\cap S_2)$ and note that $K\in{\mathcal{K}}^n_n$ as
$$S_1+S_2=H_1^{\perp}+\cdots+H_k^{\perp}=\R^n.$$
Let $j\in \{1,\dots,k\}$ and suppose without loss of generality that $j\in I_1$.  Since $H_j^{\perp}$ is orthogonal to $S_2$, we have $K\cap (H_j^{\perp}+x)=B^n\cap (H_j^{\perp}+x)$, if $x|S_2\in B^n\cap S_2$, and $K\cap (H_j^{\perp}+x)=\emptyset$, otherwise.  Therefore $K$ is rotationally symmetric with respect to $H_j$ and since $K$ is not a ball, (iv) is false.  If $H_1^{\perp}+\cdots+H_k^{\perp}=S\neq\R^n$, then $K=(B^n\cap S)\times (B^n\cap S^{\perp})\in{\mathcal{K}}^n_n$ is not a ball and again (iv) is false.
\end{proof}

\begin{cor}\label{cor22June}
Let $1\leq i\leq n-2$ and let $k=\lceil n/(n-i)\rceil$. There exist $H_j\in{\mathcal{G}}(n,i)$, $j=1,\dots,k$, such that the statements in Theorem~\ref{new} hold.
\end{cor}

\begin{proof}
It suffices to ensure that Theorem~\ref{new}(i) holds.  This is easily accomplished by choosing $H_j\in{\mathcal{G}}(n,i)$ inductively such that for $k=1,2,\dots$, we have
$H_{k+1}^\perp\not\subset(H_1^\perp+\dots +H_k^\perp)^\perp$ and
$$\dim(H_1^{\perp}+\cdots+H_k^{\perp})\ge \min\{k(n-i),n\},$$
until $k(n-i)\ge n$, i.e., $k=\lceil n/(n-i)\rceil$.
\end{proof}

In the narrower context of surfaces in $\R^3$, the first statement in the following simple corollary was observed without proof in \cite[p.~10]{HC} and proved in \cite[Theorem~A]{KK}.

\begin{cor}\label{corSep26}
Let $n\geq3$. If $H_1, H_2\in{\mathcal{G}}(n,1)$ with $H_1\neq H_2$, the statements in Theorem~\ref{new} hold.

No finite $k$ exists such that if $H_j$, $j=1,\dots,k$, are different subspaces with $2\le \dim H_j\le n-2$ for $j=1,\dots,k$, the statements in Theorem~\ref{new} hold.
\end{cor}

\begin{proof}
Let $H_1, H_2\in{\mathcal{G}}(n,1)$ with $H_1\neq H_2$. Since $\dim H_1^{\perp}=\dim H_2^{\perp}=n-1$ and $H_1^{\perp}\neq H_2^{\perp}$, Theorem~\ref{new}(i) holds.

Suppose that $k\in \N$. The assumption that $2\le \dim H_j\le n-2$ for $j=2,\dots,k$ allows us to choose different $H_j$'s such that $H_1^{\perp},\dots, H_k^{\perp}$ are all contained in the same $(n-1)$-dimensional subspace. The first condition in Theorem~\ref{new}(i) then fails.

With these observations in hand, the result follows from Theorem~\ref{new}.
\end{proof}

\section{Open problems}\label{problems}

\begin{prob}\label{probaug1}
For $i\in \{1,\dots,n-2\}$, find necessary and sufficient conditions for a set $\mathcal{F}$ of $i$-dimensional subspaces such that reflection symmetry with respect to each subspace in $\mathcal{F}$ implies full rotational symmetry.
\end{prob}

As was mentioned after Theorem~\ref{teo:sphericalsym_lines}, for $i=n-1$ such conditions are a consequence of the results in \cite[Proposition~4.2]{BCD} and \cite{EP}.

\begin{prob}\label{proboct16}
Let $i\in \{1,\dots,n-1\}$, and let $H_j\in {\mathcal{G}}(n,i)$, $j=1,\dots,k$, be such that reflection symmetry with respect to each $H_j$ implies full rotational symmetry.  Is it true that reflection symmetry with respect to each $H_j^{\perp}$ implies full rotational symmetry?
\end{prob}

As in Remark~\ref{remoct16}, a positive answer to Problem~\ref{proboct16} for $o$-symmetric sets follows from \eqref{reflection_wrt_perp}, but in general the latter equality does not help in any simple way.  For example, if $E$ is the set of vertices of a regular simplex $T$ inscribed in $S^2$, and $H_j\in {\mathcal{G}}(n,1)$, $j=1,\dots,4$, are the lines through the midpoints of opposite pairs of edges of $T$, then $R_{H_j}E=E$, $j=1,\dots,4$, but $R_{H_j^{\perp}}E\neq E$, $j=1,\dots,4$.

\end{document}